\documentclass[12pt]{amsart}
\usepackage[latin1]{inputenc}
\usepackage{color}
\usepackage{bbm}
\usepackage{amsmath,amssymb}
\usepackage{enumerate}
\usepackage{mathrsfs}
\usepackage{verbatim}
\usepackage{graphicx}

\newtheorem{thm}{Theorem}[section]

\newtheorem{prop}[thm]{Proposition}
\theoremstyle{definition}

\numberwithin{equation}{section}

\makeatletter
\newcommand{\rd}{\,\mathrm{d}}

\newcommand{\bsx}{\boldsymbol{x}}

\newcommand{\bsz}{\boldsymbol{z}}

\newcommand{\bst}{\boldsymbol{t}}

\newcommand{\bszero}{\boldsymbol{0}}
\newcommand{\bsone}{\boldsymbol{1}}

\newcommand{\NN}{\mathbb{N}}

\newcommand{\cH}{\mathcal{H}}
\newcommand{\ZZ}{\mathbb{Z}}

\newcommand{\disp}{\mathrm{disp}}

\newcommand{\cP}{\mathcal{P}}

\newcommand{\vecs}{\boldsymbol{s}}
\newcommand{\vecp}{\boldsymbol{p}}

\allowdisplaybreaks

\begin{document}
\title{Dispersion of digital $(0,m,2)$-nets}
\author[R. Kritzinger]{Ralph Kritzinger}
\address{Institute of Financial Mathematics and Applied Number Theory, Johannes Kepler University Linz, Austria, 4040 Linz, Altenberger Strasse 69}
\email{ralph.kritzinger@jku.at}

\begin{abstract}
   We study the dispersion of digital $(0,m,2)$-nets; i.e. the size of the largest axes-parallel box within such point sets. Digital nets are an important class of low-discrepancy point sets. We prove tight lower and upper bounds for certain subclasses of digital nets where the generating matrices are of triangular form and compute the dispersion of special nets such as the Hammersley point set exactly.
\end{abstract}

\maketitle 
\section{General facts on dispersion of point sets in the plane}
Let $N\in\NN$. Given an $N$-element point set in $[0,1]^2$, we define its dispersion to be the volume of the largest empty axes-parallel box amidst the point set.
More precisely, define $\mathcal{B}:=\{[x_1,y_1)\times [x_2,y_2)\mid 0\leq x_1\leq y_1\leq 1, 0\leq x_2\leq y_2\leq 1\}$. Then the dispersion $\disp(\cP)$ of the point set $\cP$ is given by $\disp(\cP):=\sup_{B\in\mathcal{B}, B\cap\cP=\emptyset}\lambda(B)$, where $\lambda(B)$ denotes the area of the box $B$.
 For the dispersion of any $N$-element point set $\cP_N$ in $[0,1]^2$ it is known that
 $$ \mathrm{disp}(\cP_N)\geq \max\left\{\frac{1}{N+1},\frac{5}{4(N+5)}\right\} $$
(see~\cite{Dumi}).
The best known construction with respect to dispersion is the Fibonacci lattice $\mathcal{F}_m$ for $m\geq 6$, for which we have
$$ \mathrm{disp}(\mathcal{F}_m)=\frac{2(F_m-1)}{F_m^2}, $$
where $F_m=\#\mathcal{F}_m$ is the $m$-th Fibonacci number (see~\cite{Bren}). So no better dispersion than asymptotically $2/N$ is known for point sets in the plane, not
even existence results. Therefore we have $\liminf_{N\to\infty} N\disp(N,2)\in [\frac54,2]$, where $\disp(N,2):=\inf_{\cP\subset [0,1]^2:\, \#\cP=N}\disp(\cP)$. It is an interesting open problem to find the exact value of this limes inferior.\\
The study of the behaviour of the dispersion of point sets in high dimensions has led to many recent papers. It is known that $\disp(\cP)\geq \frac{c_d}{N}$ for all $N$-element point sets in the $d$-dimensional unit cube $[0,1]^d$, where $c_d>0$ is independent of $N$ and tends to infinity at least logarithmically with $d$ (see~\cite[Theorem 1]{Aist}). On the other hand, we have the upper bound $\disp(\cP)\leq \frac{2^{7d}}{N}$, which is attained for certain $(t,m,d)$-nets (see~\cite[Section 4]{Aist}). We refer to~\cite{Ullr} for another upper bound on the dispersion  obtained by probabilistic arguments and for a short survey on further known results and applications of dispersion.\\
This paper is dedicated to a thorough study of the dispersion of digital $(0,m,2)$-nets, an important class of low-discrepancy point sets in the plane. It is known that the dispersion of $(0,m,2)$-nets is asymptotically $c/N$  for some positive constant $c$ and where $N=2^m$ is the number of elements of the net (see~\cite[Prop. 3.1]{Rote}). 
We will investigate how large the constant $c$ must necessarily be and how large it can be at most for certain instances of such point sets. \\
The structure of the paper is as follows: In Section 2 we give the definition of $(0,m,2)$-nets in base $b$ and prove a general upper bound on their dispersion. In the following sections we concentrate on digital $(0,m,2)$-nets which are generated by matrices as explained in Section 3. The focus will be on special digital nets which are generated by matrices with triangular form which we call NUT and NLT nets. For these classes of digital nets we will prove lower and upper bounds on their dispersion in Section 3. In Section 4, we will prove exact formulas or at least very good lower bounds for the dispersion of important instances of NUT and NLT nets. In particular, we will observe that our bounds on NUT and NLT nets as proven in Section 3 are basically sharp. In the last section, we compare our results on dispersion with known results on discrepancy of digital nets.

\section{A first upper bound for arbitrary $(0,m,2)$-nets}

Let $m\in\NN$. A $(0,m,2)$-net in base $b$ is a point set in $[0,1]^2$ with $N=b^m$ elements such that every $b$-adic interval of the form
  $$ \left[\frac{a_1}{b^{j_1}},\frac{a_1+1}{b^{j_1}}\right)\times  \left[\frac{a_2}{b^{j_2}},\frac{a_2+1}{b^{j_2}}\right),$$
where $j_1,j_2\in \NN_0$ and $a_i \in\{0,1,\dots, b^{j_i}-1\}$ for $i\in\{1,2\}$ with area $b^{-m}$ (i.e. where $j_1+j_2=m$) contains exactly
one point of the point set.
 We immediately get the following upper bound on the dispersion of $(0,m,2)$-nets in base $b$, using an idea of Rote and Tichy as presented in~\cite[Prop. 3.1]{Rote} or Larcher~\cite[Section 4]{Aist}.
This bound shows that the dispersion of $(0,m,2)$-nets is optimal with respect to the order of magnitude of $N$.

\begin{prop}
  Let $\cP$ be a $(0,m,2)$-net in base $b$. Then we have $\mathrm{disp}(\cP)\leq \frac{4b^2}{b^m}.$
\end{prop}

\begin{proof}
   Any interval $\mathcal{I}\subset [0,1]$ of length $|\mathcal{I}|$ contains a $b$-adic interval of length larger than $|\mathcal{I}|/(2b)$. Therefore, any box $B\subseteq [0,1]^2$
	 contains a dyadic box of size larger than $\lambda(B)/(2b)^2$. Thus, if $\lambda(B)/(2b)^2\geq b^{-m}$, then $B$ contains a dyadic box of size $b^{-m}$
	and therefore a point of $\cP$. This yields $$\mathrm{disp}(\cP)\leq \frac{4b^2}{b^m}.$$
\end{proof}

We can expect that there exist much better bounds for special $(0,m,2)$-nets. The aim of this article is to find such bounds.

\section{Improved dispersion bounds for certain digital $(0,m,2)$-nets}

A digital net in prime base $b\geq 2$ is a point set $\{\boldsymbol{x}_0,\ldots, \boldsymbol{x}_{b^m-1}\}$ in $[0,1)^2$,
which is generated by two matrices of size $m\times m$ with entries in $\ZZ_b$, the field with $b$ elements. The construction is as follows. Let $m\geq 1$ be an integer.
\begin{itemize}
\item Choose a bijection $\phi:\{0,1,\dots, b-1\} \rightarrow \mathbb{Z}_b$.
\item Choose $m \times m$ matrices $C_1$ and $C_2$ over $\ZZ_b$.
\item For some $n\in\{0,1,\dots,b^m-1\}$ let $n=e_1+b e_2  +\cdots +b^{m-1}e_{m}$ with $e_i\in\{0,1,\dots,b-1\}$ for all $i\in\{1,\dots,m\}$ be the $b$-adic expansion of $n$. Map $n$ to the vector $\vec{n}=(\phi(e_1),\ldots , \phi(e_{m}))^{\top}$.
\item Compute $C_j \vec{n}=:(y_{n,1}^{(j)},\ldots ,y_{n,m}^{(j)})^{\top}$ for $j=1,2$.
\item Compute $x_n^{(j)}=\frac{\phi^{-1}(y_{n,1}^{(j)})}{b}+\cdots +\frac{\phi^{-1}(y_{n,m}^{(j)})}{b^m}$ for $j=1,2$.  
\item Set $\boldsymbol{x}_{n}=(x_n^{(1)},x_{n}^{(2)})$.
\item Repeat steps 3 to 6 for all $n\in \{0,1,\dots,b^m-1\}$ and set $\cP:=\{\bsx_0,\dots,\bsx_{b^m-1}\}$. We call $\cP$ a digital net generated by $C_1$ and $C_2$.
\end{itemize}
A digital net in base $b$ is a $(0,m,2)$-net in base $b$ if and only if the following condition holds: For every choice of integers $d_1,d_2\in \NN_0$ with $d_1+d_2=m$ the first $d_1$ lines of $C_1$ and the first $d_2$ lines of $C_2$ are linearly independent over $\ZZ_b$. We refer to~\cite{Dick} for a comprehensive introduction to digital nets. \\
First we show a trivial lower bound on the dispersion of digital $(0,m,2)$-nets in base $b$, which demonstrates that we cannot get (asymptotic) dispersion $c/N$ with a constant $c<2$ for digital $(0,m,2)$-nets.
\begin{prop} \label{trivlower}
  For any digital $(0,m,2)$-net in base $b$ we have  $\mathrm{disp}(\cP)\geq \frac{2}{b^m}\left(1-\frac{1}{b^m}\right)$.
\end{prop}

\begin{proof}
    Let $\cP=\{(x,y(x)):x=0,\frac{1}{b^m},\dots,\frac{b^m-1}{b^m}\}$ be a digital $(0,m,2)$-net in base $b$. Then there exists an $x\in \{\frac{1}{b^m},\dots,\frac{b^m-1}{b^m}\}$ such that
		$y(x)=1-\frac{1}{b^m}$. Hence the box $$ \left(x-\frac{1}{b^m},x+\frac{1}{b^m}\right)\times \left(0,1-\frac{1}{b^m}\right) $$
		 of size $\frac{2}{b^m}\left(1-\frac{1}{b^m}\right)$ is empty and the result follows.
\end{proof}

For $C_1$ we will always choose the following matrix (in the case $b=2$ this choice means no loss of generality)

\begin{equation*} \label{matrixa} C_1=J_m:=
\begin{pmatrix}
0 & 0 & 0 & \cdots & 0 & 0 & 1 \\
0 & 0 & 0 & \cdots & 0 & 1 & 0 \\
0 & 0 & 0 & \cdots & 1 & 0 & 0 \\
\vdots & \vdots & \vdots & \ddots & \vdots & \vdots & \vdots & \\
0 & 0 & 1 & \cdots & 0 & 0 & 0 \\
0 & 1 & 0 & \cdots & 0 & 0 & 0 \\
1 & 0 & 0 & \cdots & 0 & 0 & 0 \\
\end{pmatrix}.
\end{equation*}

We speak of a NUT net, if $C_2$ is a nonsingular upper (right) triangle matrix; i.e.
\begin{equation} \label{NUTmatrix}
 C_2=
\begin{pmatrix}
c_{1,1} & c_{1,2} & c_{1,3} & \cdots & c_{1,m-2} & c_{1,m-1} & c_{1,m} \\
0 & c_{2,2} & c_{2,3} & \cdots &  c_{2,m-2} & c_{2,m-1} & c_{2,m} \\
0 & 0 & c_{3,3} & \cdots & c_{3,m-2} & c_{3,m-1} & c_{3,m} \\
\vdots & \vdots & \vdots & \ddots & \vdots & \vdots & \vdots & \\
0 & 0 & 0 & \cdots &  c_{m-2,m-2} & c_{m-2,m-1} & c_{m-2,m} \\
0 & 0 & 0 & \cdots &  0 & c_{m-1,m-1} & c_{m-1,m} \\
0 & 0 & 0 & \cdots &  0 & 0 & c_{m,m} \\
\end{pmatrix}.
\end{equation}
and of a NLT net, if $C_2$ is a nonsingular lower (left) triangle matrix; i.e.
\begin{equation*}
 C_2=
\begin{pmatrix}
c_{1,1} & 0 & 0 & \cdots & 0 & 0 & 0 \\
c_{2,1} & c_{2,2} & 0 & \cdots &  0 & 0 & 0 \\
c_{3,1} & c_{3,2} & c_{3,3} & \cdots & 0 & 0& 0 \\
\vdots & \vdots & \vdots & \ddots & \vdots & \vdots & \vdots & \\
c_{m-2,1} & c_{m-2,2} & c_{m-2,1} & \cdots &  c_{m-2,m-2} &0 & 0 \\
c_{m-1,1} & c_{m-1,2} & c_{m-1,1} & \cdots &  c_{m-1,m-2} & c_{m-1,m-1} & 0\\
c_{m,1} & c_{m,2} & c_{m,1} & \cdots &  c_{m,m-2} & c_{m,m-1} & c_{m,m} \\
\end{pmatrix}.
\end{equation*}
Clearly, the diagonal entries of such matrices must be non-zero.
In the following, if $n=e_1+be_2+\dots+b^{m-1}e_m$ with $e_i\in\{0,1,\dots,b-1\}$ for all $i\in \{1,\dots,m\}$ is the $b$-adic expansion of $n$, we shall write $n=[e_me_{m-1}
\dots e_1]_b$.

\begin{thm} \label{NUT}
  Let $\cP$ be a digital $(0,m,2)$ NUT or NLT net in base $b$. Then $$\mathrm{disp}(\cP)\leq \frac{2b}{b^m}\left(1-\frac{b}{2b^m}\right).$$
\end{thm}

\begin{proof}
  We first show the result for NLT nets. Let $\cP$ be such a point set, which can be written as
	$$ \cP=\left\{\left(\frac{n}{b^m},\mu_m(n)\right):n=0,1,\dots,b^m-1\right\}, $$
	where for $n\in\{0,1,\dots,b^m-1\}$ with $n=[e_m\dots e_1]_b$ we define 
	$$\mu_m(n)=\frac{c_{1,1}\odot e_1}{b}+\frac{c_{2,1}\odot e_1\oplus c_{2,2}\odot e_2}{b^2}+\dots+\frac{c_{m,1}\odot e_1\oplus \dots \oplus c_{m,2}\odot e_m}{b^m}.$$
	The operators $\oplus$ and $\odot$ denote addition and multiplication modulo $b$, respectively.
  For $k\in\{1,\dots,m-1\}$ consider a subset $A_k$ of $\{1,\dots,b^m-1\}$ of $b^k$ consecutive numbers. Let $n\neq n'$ be distinct elements of $A_k$ and $n=[e_m\dots e_{k+1}e_k\dots e_1]_b$ and $n'=[e_m'\dots e_{k+1}'e_k'\dots e_1']_b$ their $b$-adic expansions. Then clearly the strings $(e_k \dots e_1)$ and $(e_k' \dots e_1')$ are distinct. Since the first $k$ terms of $\mu_m(n)$ depend only on the digits $e_1,\dots,e_k$ of $n$ and the diagonal entries $c_{i,i}$ are non-zero for all $i\in\{1,\dots,m\}$, we conclude that 
	  $$ \{\mu_m(n)\mid n \in A_k\}=\left\{\frac{s}{b^k}+\zeta_s: s=0,1,\dots,b^k-1\right\}, $$
		where $0\leq \zeta_s \leq b^{-k}-b^{-m}$ for all $s\in\{0,1,\dots,b^k-1\}$. Therefore the largest gap between two consecutive numbers in $\{\mu_m(n)\mid n \in A_k\}\cup \{0,1\}$
	 is at most $2b^{-k}-b^{-m}$. As a consequence, the area of empty boxes $B=I_x \times I_y$ such that $b^m |I_x| \in\{b^{k}+1,\dots, b^{k+1}\}$ is bounded by
	 $$ |B|\leq \frac{b^{k+1}}{b^m} \left(\frac{2}{b^k}-\frac{1}{b^m}\right)=\frac{2b}{b^m}-\frac{b^{k+1}}{b^{2m}}\leq \frac{2b}{b^m}-\frac{b^2}{b^{2m}}=\frac{2b}{b^m}\left(1-\frac{b}{2b^m}\right),$$
	whereas boxes of length $b^m|I_x|\leq b$ have area less than $\frac{b}{b^m}$. \\
	We employ the following simple observation to derive the result also for NUT nets. Let such a net $\cP'$ be generated by $J_m$ and a NUT matrix $C_2$.
	If we multiply both matrices with the same regular matrix from the right side,
the new matrices generate the same point set (see e.g.~\cite[Lemma 4.61]{Dick}). Hence, the matrices $J_m C_2^{-1} J_m$ and $J_m$
generate $\cP'$. Note that $J_m C_2^{-1} J_m$ is a NLT matrix. Further, if we switch the roles of $J_m C_2^{-1} J_m$ and $J_m$, then the matrices generate $\{(y,x)\mid (x,y)\in \cP'\}$. Clearly the latter
set of points has same dispersion as $\cP'$, as for every empty box $B=I_x \times I_y$ in $\cP'$ we can find an empty box of same size
in $\{(y,x)\mid (x,y)\in \cP'\}$, namely $B'=I_y \times I_x$. Hence, for every NUT net we can find a corresponding NLT net with same dispersion (and the other way round), which yields the result for NUT nets.
\end{proof}

Note that for NUT or NLT nets in base 2 this theorem implies $\mathrm{disp}(\cP)\leq \frac{4}{2^m}$ for all $m\in\NN$. We will see in the subsequent section that the constant $4$ in this upper bound cannot be replaced by a smaller constant as $m$ grows larger. \\
Next we improve the lower bound on the dispersion of digital $(0,m,2)$-nets as given in Theorem~\ref{trivlower} for NUT and NLT nets in the dyadic case $b=2$. We will see in the subsequent section that this lower bound is sharp.

\begin{thm} \label{better}
   Let $m\geq 5$ and $\cP$ be a digital $(0,m,2)$ NUT or NLT net in base 2. Then
     $$ \disp(\cP)\geq \frac52 \frac{1}{2^m}.$$
\end{thm}

\begin{proof}
  In this proof it is more convenient to show the result for NUT nets. Then it follows also for NLT nets by the arguments in the second half of the proof of Theorem~\ref{NUT}. \\
  Consider a NUT net generated by $J_m$ and the matrix $C_2$ as in~\eqref{NUTmatrix}. Such a net can also be defined via 
	$$ \cP=\left\{\left(\frac{n}{2^m},\nu_m(n)\right):n=0,1,\dots,2^m-1\right\}, $$
	where for $n\in\{0,1,\dots,2^m-1\}$ with $n=[e_m\dots e_1]_2$ we define 
	$$\nu_m(n)=\frac{e_1 \oplus c_{1,2}e_2\oplus\dots\oplus  c_{1,m}e_m}{2}+\dots+\frac{e_{m-1}\oplus c_{m-1,m} e_m}{2^{m-1}}+\frac{e_m}{2^m}.$$
	By this definition it is obvious that the numbers $\nu_m(n)$ for $n\in\{0,\dots,7\}$ depend only on the matrix entries $c_{1,2}$, $c_{1,3}$ and $c_{2,3}$. The numbers $\nu_m(n)$ for $n\in\{0,\dots,15\}$ depend only on the matrix entries $c_{1,2}$, $c_{1,3}$, $c_{2,3}$, $c_{1,4}$, $c_{2,4}$ and $c_{3,4}$.  The numbers $\nu_m(n)$ for $n\in\{0,\dots,31\}$ depend only on the six matrix entries already mentioned before and additionally on $c_{1,5}$, $c_{2,5}$, $c_{3,5}$ and $c_{4,5}$.
	Now with a software such as \textsc{mathematica} it can be checked that for every NUT net there are certain gaps between consecutive numbers in the string $(\nu_m(1),\dots,\nu_m(31))$: either there exist 4 consecutive numbers in this string such that the largest gap between these numbers (sorted from smallest to largest) is at least $\frac12$, or there exist 9 consecutive numbers such that the largest gap between these numbers is at least $\frac14$, or there exist 19 consecutive numbers such that the largest gap between these numbers is at least $\frac18$. In fact, one of these three situations occurs in most of the $2^{10}$ possible cases for the sequence $(c_{1,2},\dots,c_{4,5})$ of the 10 matrix entries mentioned above, which all imply boxes of size $\frac52\frac{1}{2^m}$. In the remaining cases, we still find either 2 consecutive numbers with distance at least $\frac{29}{32}$ or 28 consecutive numbers with a gap of at least $\frac{3}{32}$, which imply the existence of empty boxes of size $\frac{87}{32}\frac{1}{2^m}>\frac52\frac{1}{2^m}$. This completes the proof, but we would like to add a few words to the procedure: Actually, one does not really have to check all $2^{10}$ cases, as in many cases we find empty boxes of the desired size faster. For instance, if $(c_{1,2},c_{1,3},c_{2,3})=(1,1,1)$, then $(\nu_m(3),\dots,\nu_m(6))=\frac18(2,7,3,1)$,
	or if $(c_{1,2},c_{1,3},c_{2,3})=(0,1,0)$, then $(\nu_m(3),\dots,\nu_m(6))=\frac18(6,5,1,7)$, and therefore there remain only $3\cdot 2^8$ cases left to check. If $(c_{1,2},c_{1,3},c_{2,3})\in\{(0,0,1),(1,0,1),(1,1,0)\}$, then for any choice of $c_{1,4}$, $c_{2,4}$ and $c_{3,4}$ we find 4 consecutive numbers in the string $(\nu_m(1),\dots,\nu_m(15))$ such that the largest gap between these numbers is at least $\frac12$ or 9 consecutive numbers such that the largest gap is at least $\frac14$. Also for $(c_{1,2},c_{1,3},c_{2,3})\in\{(0,0,0),(0,1,1),(1,0,0)\}$ for most instances of $c_{1,4}$, $c_{2,4}$ and $c_{3,4}$ we can find such gaps, only in 10 instances this is not the case. So far these are only 48 more cases to check with the computer. For the remaining 10 instances (namely $(c_{1,2},\dots,c_{3,4})\in\{(0,0,0,0,0,0),(0,0,0,0,0,1),(0,0,0,1,1,1),(0,1,1,0,0,0),(0,1,1,0,0,1),(0,1,1,1,1,0)$,\\$(1,0,0,0,0,0),(1,0,0,0,1,1),(1,0,0,1,0,0),(1,0,0,1,0,1)\}$) we have to look at the whole string $(\nu_m(1),\dots,\nu_m(31))$ and therefore additionally at the matrix entries $c_{1,5}$, $c_{2,5}$, $c_{3,5}$ and $c_{4,5}$. So only 160 cases remain to be checked with the computer. Allover, we had 216 cases to check. Note that we proved even more: every NUT net with $m\geq 5$ contains an empty box of size at least $\frac52\frac{1}{2^m}$ in the region $[0,\frac{32}{2^m}]\times [0,1]$ and hence every NLT net contains such a box in $[0,1]\times[0,\frac{32}{2^m}] $.
\end{proof}

The upper bounds on the dispersion of $b$-adic nets suggest a worse behaviour of the constant $c$ in the dispersion bound $c/N$ as the base $b$ grows, which is not reflected in the trivial lower bound stated in Proposition~\ref{trivlower}. In the following we prove another simple lower bound on the dispersion of $b$-adic NUT and NLT nets which shows that (at least under a certain condition) the constant indeed increases with $b$.

\begin{prop}
  Let $\cP$ be a $b$-adic digital NUT $(0,m,2)$-net with $c_{1,1}=1$ or a NLT net with $c_{m,m}=1$. Then we have
  $$ \disp(\cP)\geq \max_{w=0,1,\dots,b-1} \frac{w+1}{b}\frac{b-w}{b^m}
    =\begin{cases}
        \frac12\left(\frac{b}{2}+1\right)\frac{1}{b^m} & \text{if $b$ is even}, \\
        \frac{1}{b}\left(\frac{b+1}{2}\right)^2\frac{1}{b^m} & \text{if $b$ is odd}. \\
     \end{cases}$$
\end{prop}

\begin{proof}
It is enough to show the result for NUT nets. Note that for any $b$-adic NUT net $\cP$ with $c_{1,1}=1$ we have $\nu_m(n)=\frac{n}{b}$ for $n\in\{0,1,\dots,b-1\}$.
This implies that the intervals
   $$ J_w=\left(\frac{w}{b^m},\frac{b}{b^m}\right)\times \left(0,\frac{w+1}{b}\right)  $$
  are empty for all $w\in\{0,1,\dots,b-1\}$ and the result is verified.
\end{proof}

\section{Precise dispersion of particular digital $(0,m,2)$-nets}

The Hammersley point set $\cH_{b,m}$ is a NUT (and NLT) net generated by $J_m$ as above and $C_2$ the $m\times m$ identity matrix over $\ZZ_b$. 
This construction works for any $b\geq 2$ (not only for primes). Obviously, the Hammersley point set $\cH_{b,m}$ can also be defined via
$$ \cH_{b,m}:=\left\{\left(\frac{n}{b^m},\varphi_b(n)\right):n=0,1,\dots,b^m-1\right\}, $$ 
where for $n\in\{0,1,\dots,b^m-1\}$ with $n=[e_m\dots e_1]_b$ we define 
$$\varphi_b(n)=\frac{e_1}{b}+\frac{e_2}{b^2}+\dots+\frac{e_m}{b^{m}}.$$
 We prove the following theorem on the exact dispersion of $\cH_{2,m}$. Dumitrescu and Jiang~\cite{Dumi} already found $\disp(\cH_{2,m})\leq\frac{4}{2^m}$,
which follows also from Theorem~\ref{NUT}. Rote and Tichy~\cite{Rote} proved the bound $\disp(\cP)\leq \frac{c}{N}$ for a constant $c>0$ and $N=b^m$ even earlier.  In the same paper, they also studied the dispersion of the Hammersley point set with respect to arbitrary (not necessarily axes-parallel) rectangles and found that this quantity is of exact order $\frac{1}{\sqrt{N}}$. In the following, for a set $M\subseteq \{1,\dots,2^m-1\}$ and a bijective function $\varphi: \{1,\dots,2^m-1\}\to \{1,\frac{1}{2^m},\dots,\frac{2^m-1}{2^m}\}$ we define $\varphi(M):=\{\varphi(x): x\in M\}$. Further, we call elements $x_1,x_2\in M$ consecutive in $\varphi(M)$ if $\varphi(x_1)<\varphi(x_2)$ and there is no $x^*\in M$ such that $\varphi(x_1)<\varphi(x^*)<\varphi(x_2)$. 

\begin{thm}
  For all $m\geq 4$ we have 
	  $$ \disp(\cH_{2,m})=\frac{3}{2^m}\left(1-\frac{3}{2^m}\right). $$
		Further we have $\disp(\cH_{2,1})=\frac12$, $\disp(\cH_{2,2})=\frac{3}{8}$ and $\disp(\cH_{2,3})=\frac{1}{4}$.
\end{thm}

\begin{proof}
   We only prove the case $m\geq 4$. Note that the interval $$B=\left[\frac{1}{2}-\frac{1}{2^{m-1}},\frac12+\frac{1}{2^m}\right)\times \left[\frac{1}{2^m},1-\frac{2}{2^m}\right)$$ with area $\lambda(B)=\frac{3}{2^m}\left(1-\frac{3}{2^m}\right)$ contains no point of $\cH_{2,m}$, since $\varphi(2^{m-1}-1)=1-\frac{2}{2^m}$ and $\varphi(2^{m-1})=\frac{1}{2^m}$.
	Hence $\disp(\cH_{2,m})\geq \frac{3}{2^m}\left(1-\frac{3}{2^m}\right)$. Now we show the $\leq$-part. \\
	We consider empty boxes $B=I_x \times I_y$, where $I_x, I_y \subset [0,1]$. We write $|I_x|=\ell_x$ and $|I_y|=\ell_y$. Clearly, we can assume $\ell_x, \ell_y \in \{\frac{1}{2^m},\dots,\frac{2^m-1}{2^m}, 1\}$. We prove the Theorem by looking at different instances for $\ell_x$. First, boxes with $2^m\ell_x\in\{1,2\}$ are trivially smaller than the given dispersion.\\
For $k\in\{1,\dots,m-1\}$ consider a subset $A_k$ of $\{1,\dots,2^m-1\}$ consisting of $2^{k}$ consecutive numbers.	We represent these numbers in dyadic expansion of the form $[e_m\dots e_{k+1}e_k\dots e_1]_2$, where $e_j\in\{0,1\}$ for all $j\in\{1,\dots,m\}$. Clearly, the strings $(e_k\dots e_1)$ of the last $k$ digits of the integers in $A_k$ are pairwise distinct. We consider distinct elements $x_1,x_2\in A_k$ consecutive in $\varphi(A_k)$. Then the dyadic expansions of $x_1$ and $x_2$ are of the form
\begin{align*}
     x_1=& [e_m\dots e_1]=[\vecs_1|1\dots 10 e_{i-1}\dots e_1]_2, \\
	   x_2=& [e_m' \dots e_1']=[\vecs_2|0\dots 01 e_{i-1}\dots e_1]_2,
\end{align*}
where $i\in\{1,\dots k\}$ is maximal such that $e_i=0$ and where $\vecs_1$ and $\vecs_2$ stand for the string $(e_m\dots e_{k+1})$ of the last $m-k$ digits. Such an index $i$ exists, as otherwise $\varphi(x_1)$ were already maximal in $\varphi(A_k)$. We have $\vecs_1, \vecs_2 \in \{\vecp_1,\vecp_2\}$, where $\vecp_1=(e_m\dots e_{l+1}01\dots 1)$ and $\vecp_2=(e_m\dots e_{l+1}10\dots 0)$, where $l\in\{k+1,\dots,m\}$ is minimal such that $e_l=0$. (Note that $(e_m\dots e_{k+1})=(1,\dots,1)$ implies that $(e_m'\dots e_{k+1}')=(1,\dots,1)$ as well, in which case $\varphi(x_2)-\varphi(x_1)=2^{-k}$. The situation where $(e_m\dots e_{k+1})=(1,\dots,1,0)$ and  $(e_m'\dots e_{k+1}')=(1,\dots,1,1)$ is covered in the following for $l=k+1$. So we can assume the existence of such an $l$.) Hence
$$ \varphi(x_2)-\varphi(x_1)=\frac{1}{2^k}+\sum_{j=k+1}^{m}\frac{e_j'-e_j}{2^j}. $$
We have three cases:
\begin{itemize}
   \item If $\vecs_1=\vecs_2$, i.e. $e_j'=e_j$ for all $j\in\{k+1,\dots,m\}$, then $\varphi(x_2)-\varphi(x_1)=\frac{1}{2^k}$.
	 \item If $\vecs_1=\vecp_2$ and $\vecs_2=\vecp_1$, then
	$$ \sum_{j=k+1}^{m}\frac{e_j'-e_j}{2^j}=\frac{1}{2^k}-\frac{3}{2^l} $$
	and therefore $\varphi(x_2)-\varphi(x_1)=\frac{1}{2^{k-1}}-\frac{3}{2^l}$.
	\item If  $\vecs_1=\vecp_1$ and $\vecs_2=\vecp_2$, then $\varphi(x_2)-\varphi(x_1)=\frac{3}{2^l}$.
\end{itemize}
Further there exist elements $x_1,x_2\in A_k$ with $\varphi(x_1)\leq 2^{-k}$ and $1-\varphi(x_2)\leq 2^{-k}$ (namely those elements which end with $k$ zero- or one-digits, respectively). Hence, the largest gap between two consecutive elements in the set $\{\varphi(x): x\in A_k\}\cup\{0,1\}$ is $\max\{\frac{1}{2^{k-1}}-\frac{3}{2^l},\frac{3}{2^l}\}$,
where $l\in\{k+1,\dots,m-1\}$. Assume that $2^m\ell_x\in\{2^k+1,\dots,3\,2^{k-1}\}$. Then $\ell_y\leq \max\{\frac{1}{2^{k-1}}-\frac{3}{2^l},\frac{3}{2^l}\}$ and therefore
\begin{align*} 
   |B|\leq& \frac{3}{2^m}\left(1-\frac{3}{2^m}\right).
\end{align*}

Now we consider a subset $B_k=\{x_1,\dots, x_{3\cdot 2^{k-1}}\}$ of $\{0,1,\dots,2^m-1\}$ consisting of $3\cdot 2^{k-1}$ consecutive numbers, where $k\in\{3,\dots,m-1\}$.
We need a tighter bound for the largest gap of two consecutive numbers in $\varphi(B_k)$. 
Let $\bar{B}_k=\{x_1,\dots,x_{2^k}\}$. Let $y_1, y_2\in \bar{B}_k$ consecutive in $\varphi(\bar{B}_k)$ such that $\varphi(y_2)-\varphi(y_1)=\frac{1}{2^{k-1}}-\frac{3}{2^l}$. Then $y_1=[\vecp_2|1\dots 10 e_{i-1}\dots e_1]_2$ and $y_2= [\vecp_1|0\dots 01 e_{i-1}\dots e_1]_2$. Since $y_1-y_2=2^{k+1}-3\cdot 2^{i-1}$, we must have $i=k$; i.e. actually
\begin{align*}
     y_1=& [\vecp_2|0 e_{i-1}\dots e_1]_2, \\
	   y_2=& [\vecp_1|1 e_{i-1}\dots e_1]_2.
\end{align*}
Set $y^*=[\vecp_2|1 e_{i-1}\dots e_1]_2$. We have $\varphi(y_1)<\varphi(y^*)<\varphi(y_2)$ and $y_2<y_1<y^*$. Since $y^*-y_1=2^{k-1}$, we have $y^*\in B_k$. Therefore, any gap in $\varphi(B_k)$ of length $\frac{1}{2^{k-1}}-\frac{3}{2^l}$ does no longer exist, and since $\varphi(y_2)-\varphi(y^*)<\frac{1}{2^k}$ as well as $\varphi(y^*)-\varphi(y_1)=\frac{1}{2^k}$, the longest possible gaps in $\varphi(B_k)$ are of length $\frac{3}{2^{k+1}}$. This proves the Theorem for empty boxes $B$ such that $2^m \ell_x \in\{3\cdot 2^{k-1},\dots, 2^{k+1}-3\}$ with $k\in\{3,\dots,m-1\}$. \\
Next consider a subset $C_k=\{x_1,\dots, x_{2^{k+1}-3}\}$ of $\{0,1,\dots,2^m-1\}$ consisting of $2^{k+1}-3$ consecutive numbers, where $k\in\{3,\dots,m-1\}$.
Let $\bar{C}_k=\{x_1,\dots,x_{2^k}\}$. Let $z_1, z_2\in \bar{C}_k$ consecutive in $\varphi(\bar{C}_k)$ such that $\varphi(z_2)-\varphi(z_1)=\frac{3}{2^{k+1}}$. This happens if $z_1$ and $z_2$ are of the form
\begin{align*}
     z_1=& [e_m\dots e_{k+2}0|1\dots 10 e_{i-1}\dots e_1]_2, \\
	   z_2=& [e_m\dots e_{k+2}1|0\dots 01 e_{i-1}\dots e_1]_2.
\end{align*}
Then with $z^*=[e_m\dots e_{k+2}1|1\dots 10 e_{i-1}\dots e_1]_2$ we have $\varphi(z_1)<\varphi(z^*)<\varphi(z_2)$ and $z_1<z_2<z^*$. Since $z^*-z_2=2^{k}-3\cdot 2^{i-1}\leq 2^k-3$, we have $z^*\in C_k$. Hence, the longest gap in $\varphi(C_k)$ is bounded by $\frac{1}{2^k}$, and the result follows for empty boxes $B$ such that $2^m \ell_x \in\{2^{k+1}-2,2^{k+1}-1,2^{k+1}\}$ with $k\in\{3,\dots,m-1\}$.

	Consider three consecutive numbers $x_1<x_2<x_3$ in $\{1,\dots,2^m-1\}$. The difference $\varphi(x_2)-\varphi(x_1)$ can only be larger than $\frac12$ if $x_1$ and $x_2$ are of the form
	$x_1=[e_m\dots e_3 0|0]$ and $x_2=[e_m\dots e_3 1|1]$, but then $x_2-x_1=3$ - a contradiction. Hence the result follows for empty boxes with $2^m\ell_x=4$.
	Consider six consecutive numbers $x_1<\dots<x_6$ in $\{1,\dots,2^m-1\}$. A gap in $\varphi(\{x_1,\dots,x_4\})$ larger than $\frac14$ can only occur in three cases: If $z_1,z_2 \in \{x_1,\dots,x_4\}$ are distinct, then this can only happen if $z_1=[e_m\dots e_40|00]$ and $z_2=[e_m\dots e_41|10]$ or $z_1=[e_m\dots e_40|01]$ and $z_2=[e_m\dots e_41|11]$ or $z_1=[e_m\dots e_40|10]$ and $z_2=[e_m\dots e_41|01]$. The first two cases are impossible since $|z_2-z_1|=6$, whereas in the third case the gap between $\varphi(z_1)$ and $\varphi(z_2)$ can be closed with $z^*=[e_m\dots e_41|10]$. Since $z^*-z_2=1$, we have $z^*\in \{x_1,\dots,x_6\}$. The largest gap in $\varphi(\{0,x_1,\dots,x_6,1\})$ is therefore bounded by $\frac14$, and the result follows for empty boxes with $2^m \ell_x\in \{7,8\}$. But that completes the proof.
\end{proof}

We show a lower bound on the dispersion of the $b$-adic Hammersley point set. We conjecture that we have inequality in Theorem~\ref{badic}.

\begin{thm} \label{badic}
  Let $b\geq 2$ and $m\geq 2$. Then we have
	   \begin{align*}\disp(\cH_{b,m})\geq& \frac{1}{b^{2m}}\max_{w=2,\dots,b}(w+1)\left((b-w+2)b^{m-1}-b-1\right) \\ 
		     =&\begin{cases}
				    \frac{3}{2^m}\left(1-\frac{3}{2^m}\right) & \text{if $b=2$,} \\
			      \frac{b+2}{2b^m}\left(\frac{b+4}{2b}-\frac{b+1}{b^m}\right) & \text{if $b\geq 4$ is even,} \\
						\frac{b+3}{2b^m}\left(\frac{b+3}{2b}-\frac{b+1}{b^m}\right) & \text{if $b$ is odd.}
						\end{cases}
			  \end{align*}.
\end{thm}

\begin{proof}
   We prove that for every $w\in\{2,\dots,b\}$ the interval
	  $$ \left(\frac{1}{b}-\frac{w}{b^m},\frac{1}{b}+\frac{1}{b^m}\right)\times\left(\frac{1}{b^m},\frac{b-w+2}{b}-\frac{b}{b^m}\right) $$
		contains no element of $\cH_{b,m}$, which implies the lower bound. To this end, we show that for all $n\in\{b^{m-1}-w+1,\dots,b^{m-1}\}$ we have $\varphi_b(n)\notin \left(\frac{1}{b^m},\frac{b-w+2}{b}-\frac{b}{b^m}\right)$. Clearly $\varphi_b(b^{m-1})=\frac{1}{b^{m}}$. Let $s\in\{1,\dots,w-1\}$. Then
		\begin{align*}
		    \varphi_b(b^{m-1}-s)=&\varphi_b\left((b-s)+\sum_{j=1}^{m-2}(b-1)b^j\right)=\frac{b-s}{b}+\sum_{j=2}^{m-1}\frac{b-1}{b^j}\geq \frac{b-w+2}{b}-\frac{b}{b^m}.
		\end{align*}
		The proof is complete.
\end{proof}

Now we will show that there exists a digital $(0,m,2)$-net in base 2 with essentially lower asymptotic dispersion than the dyadic Hammersley point set. Our candidate is the net $\cP_m^{U}$ generated by the matrices
\begin{equation*} \label{matrixa1} C_1=J_m 
\text{\, and\,}
 C_2=U_m:=
\begin{pmatrix}
1 & 1 & 1 & \cdots & 1 & 1 & 1 \\
0 & 1 & 1 & \cdots & 1 & 1 & 1 \\
0 & 0 & 1 & \cdots & 1 & 1 & 1 \\
\vdots & \vdots & \vdots & \ddots & \vdots & \vdots & \vdots & \\
0 & 0 & 0 & \cdots &  1 & 1 & 1 \\
0 & 0 & 0 & \cdots &  0 & 1 & 1 \\
0 & 0 & 0 & \cdots &  0 & 0 & 1 \\
\end{pmatrix}.
\end{equation*}
The following result on the dispersion of $\cP_m^{U}$ demonstrates that the lower bound in Theorem~\ref{better} is sharp.
\begin{thm}
  For all $m\geq 4$ we have 
	  $$ \disp(\cP_m^{U})=\frac52\frac{1}{2^m}. $$
\end{thm}

\begin{proof}
Given a digital $(0,m,2)$-net generated by the matrices $J_m$ and $U_m$. By the arguments in the proof of Theorem~\ref{NUT} we find that
$\cP_m^{U}$ has same dispersion as the NLT net $\tilde{\cP}_m^{U}$ generated by $J_m$ and the matrix
\begin{equation*} \label{matrixa2} J_mU_m^{-1}J_m=
\begin{pmatrix}
1 & 0 & 0 & \cdots & 0 & 0 & 0 \\
1 & 1 & 0 & \cdots & 0 & 0 & 0 \\
0 & 1 & 1 & \cdots & 0 & 0 & 0 \\
\vdots & \vdots & \vdots & \ddots & \vdots & \vdots & \vdots & \\
0 & 0 & 0 & \cdots &  1 & 0 & 0 \\
0 & 0 & 0 & \cdots &  1 & 1 & 0 \\
0 & 0 & 0 & \cdots &  0 & 1 & 1 \\
\end{pmatrix}.
\end{equation*}
Hence, we investigate the dispersion of the net generated by $J_m$ and $J_mU^{-1}J_m$, which can be written as
$$ \tilde{\cP}_m^{U}=\left\{\left(\frac{n}{2^m},\psi(n)\right):n=0,1,\dots,2^m-1\right\}, $$
	where for $n=[e_m\dots e_1]$ we set $\psi(n)=\frac{e_1}{2}+\frac{e_1\oplus e_2}{2^2}+\dots+\frac{e_{m-1}\oplus e_m}{2^{m-1}}. $
  Thus, we want to show $\disp(\tilde{\cP}_m^{U})=\frac52\frac{1}{2^m}$.
		For the lower bound, we can either refer to Theorem~\ref{better} or observe that the box
  $$ \left(\frac{3}{2^m},\frac{7}{2^m}\right) \times \left(\frac{5}{16},\frac{15}{16}\right)  $$
	is empty, since $\psi(4/2^m)=3/16$, $\psi(5/2^m)=15/16$ and $\psi(6/2^m)=5/16$. \\
	We prove the upper bound. Our approach will be as follows. For $k\in\{1,\dots,m-1\}$ let $A_k$, $B_k$ and $C_k$ be sets consisting of $2^k$, $\lfloor \frac43 2^k \rfloor$ and $\frac32 2^k$ consecutive elements of $\{1,\dots,2^m-1\}$, respectively. We will show that the largest gaps of consecutive elements in $\psi(A_k)\cup\{0,1\}$, $\psi(B_k)\cup\{0,1\}$ and $\psi(C_k)\cup\{0,1\}$ are bounded by $\frac{7}{2^{k+2}}$, $\frac{3}{2^{k+1}}$ and $\frac{5}{2^{k+2}}$, respectively. 
	This yields the result. \\
 We consider two elements $x_1,x_2\in A_k$ such that $\psi(x_1)$ and $\psi(x_2)$ are consecutive elements in $\psi(A_k)$. We write
		     \begin{align*}
     x_1=& [e_{m}\dots e_{k+1}e_k\dots\dots e_1]_2, \\
	   x_2=& [e'_{m}\dots e'_{k+1}e'_k\dots\dots e'_1]_2.
\end{align*}   
  Hence 
	  $$ \psi(x_1)-\psi(x_2)=\frac{1}{2^k}+\sum_{j=k+1}^{m}\frac{(e'_{j-1}\oplus e'_j)-(e_{j-1}\oplus e_j)}{2^j}=:\frac{1}{2^k}+E_k. $$
		Note that we have four possibilities how the strings $(e_m\dots e_k)$ and $(e'_m\dots e'_k)$ might look, namely
		\begin{align*} \vecp_{11}=&(e_m\dots e_{l+1}0|1\dots 1|0), \, \vecp_{12}=(e_m\dots e_{l+1}0|1\dots 1|1), \\  \vecp_{21}=&(e_m\dots e_{l+1}1|0\dots 0|0), \, \vecp_{22}=(e_m\dots e_{l+1}1|0\dots 0|1). \end{align*}
		Here we have $l\in\{k+1,\dots,m\}$ minimal such that $e_l=0$.
		Assume that $(e_m\dots e_k)=\vecp_i$ and $(e'_m\dots e'_k)=\vecp_j$ for some $i,j\in\{11,12,21,22\}$. 
		Then it is straightforward to check that we get for the difference $\Delta \psi:=\psi(x_2)-\psi(x_1)$ the value $\frac{3}{2^{k+1}}+\frac{1}{2^{l+1}}\leq \frac{7}{2^{k+2}}$ in the following cases: For $l\geq k+2$ this happens if $(i,j)=(21,11)$ and $e_{l+1}=1$ or if $(i,j)=(12,22)$ and $e_{l+1}=0$, whereas for $l=k+1$ this difference occurs if $(i,j)=(11,21)$ and $e_{l+1}=0$ or if $(i,j)=(22,12)$ and $e_{l+1}=1$. Otherwise $\Delta\psi$ is smaller or equal $\frac{3}{2{k+1}}$ and matches this bound only for $(i,j)=(11,12)$ or $(i,j)=(22,21)$ and $l= k+1$. In all other cases we have $\Delta\psi\leq \frac{1}{2^k}+\frac{1}{2^{l+1}}\leq\frac{5}{2^{k+2}}$. Allover, we find that the largest gap between two consecutive elements in $\psi(A_k)\cup\{0,1\}$ is bounded by $\frac{7}{2^{k+2}}$. \\
Next we consider a set $B_k=\{x_1,\dots,x_{\lfloor \frac43 2^k \rfloor}\}$ consisting of $\lfloor \frac43 2^k \rfloor$ consecutive elements of $\{1,\dots,2^m-1\}$. We show that in $\psi(B_k)\cup\{0,1\}$ larger gaps than $\frac{3}{2^{k+1}}$ do no longer occur. Consider the set $\bar{B}_k=\{x_1,\dots,x_{2^k}\}$ and let $y_1, y_2\in\bar{B}_k$ consecutive in $\psi(\bar{B}_k)$. As observed above, such gaps can only occur if either 
		 \begin{align*}
     y_1=& [e_{m}\dots e_{l+2}xx|\tilde{x}\dots \tilde{x}|\tilde{x}x\dots \tilde{x}x|x e_{i-2} \dots e_1]_2, \\
	   y_2=& [e_{m}\dots e_{l+2}x\tilde{x}|x\dots x|\tilde{x}\tilde{x}\dots \tilde{x}\tilde{x}|x e_{i-2} \dots e_1]_2,
\end{align*} 
or if
	\begin{align*}
     y_1=& [e_{m}\dots e_{k+3}xx|x\tilde{x}\dots x\tilde{x}|\tilde{x} e_{i-2} \dots e_1]_2, \\
	   y_2=& [e_{m}\dots e_{k+3}x\tilde{x}|xx\dots xx|\tilde{x} e_{i-2} \dots e_1]_2,
\end{align*} 
			where $x\in\{0,1\}$ and $\tilde{x}=1-x$ and where $i\in\{1,\dots,k\}$ is maximal such that $e_i\oplus e_{i-1}=0$. With $y_1$ and $y_2$ given like in the first of these two situations, we find $$|y_2-y_1|=\frac16 (2^{k+3}-2^i)>2^k;$$ i.e. not both numbers can be in $\tilde{B}_k$. So this case does not occur. However, the second case is valid since $|y_2-y_1|<2^k$. We have $\Delta\psi=\frac{7}{2^{k+2}}$. However, we can close this gap with a particular element $y^*$ in $B_k$. If $x=0$ we choose $y^*=[e_{m}\dots e_{k+3}01|01\dots 01|1 e_{i-2} \dots e_1]_2$, for which we have $\psi(y_1)<\psi(y^*)<\psi(y_2)$, $\psi(y_2)-\psi(y^*)=\frac{1}{2^k}$ and $\psi(y^*)-\psi(y_1)=\frac{3}{2^{k+2}}$.
			Further, we have $y_1<y_2<y^*$ and
			$$ y^*-y_2=\sum_{j=0}^{(k-i-1)/2}2^{i+2j-1}=\frac16 (2^{k+1}-2^i)\leq \left\lfloor \frac13 2^k\right\rfloor, $$
			which shows that indeed $y^*\in B_k$. If $x=1$, then $y^*=[e_{m}\dots e_{k+3}11|11\dots 11|0 e_{i-2} \dots e_1]_2$ does the trick analogously.\\
	Next we consider a set $C_k=\{x_1,\dots,x_{3\cdot 2^{k-1}}\}$ consisting $3\cdot 2^{k-1} $ consecutive elements of $\{1,\dots,2^m-1\}$. We show that in $\psi(C_k)\cup\{0,1\}$ larger gaps than $\frac{5}{2^{k+2}}$ do no longer occur. Consider the set $\bar{C}_k=\{x_1,\dots,x_{2^k}\}$ and let $z_1, z_2\in\bar{C}_k $. As observed above, gaps larger than $\frac{5}{2^{k+2}}$, namely of length $\frac{3}{2^{k+1}}$, can only occur if 
	\begin{align*}
     z_1=& [e_{m}\dots e_{k+2}x|x\tilde{x}\dots x\tilde{x}x|x e_{i-2} \dots e_1]_2, \\
	   z_2=& [e_{m}\dots e_{k+2}x|\tilde{x}\tilde{x}\dots \tilde{x}\tilde{x}\tilde{x}|x e_{i-2} \dots e_1]_2,
\end{align*} 
		where $l=k+1$, $x\in\{0,1\}$ and $\tilde{x}=1-x$. If $x=0$ choose 
		$$z^*= [e_{m}\dots e_{k+2}1|01\dots 010|0 e_{i-2} \dots e_1]$$ 
		to close the gap. If $x=1$, then we know that $k\leq m-2$ and that there must be an $i\in\{k+2,\dots,m\}$ such that $e_i=0$, because otherwise $z_1$ could not be in $\bar{C}_k$. Let $p$ be the minimal index in $i\in\{k+2,\dots,m\}$ with $e_i=0$. Then the number 
		$$z^*=[e_{m}\dots e_{p+1}10\dots 0|00\dots 00|1 e_{i-2} \dots e_1]$$ does the trick. In both cases we have $z^*-\max\{z_1,z_2\}\leq \frac16(2^{k+1}+2^i)\leq 2^{k-1}$, and therefore $z^*\in C_k$. 
		 Hence, the longest gap between consecutive numbers in $\psi(C_k)\cup \{0,1\}$ is bounded by $\frac{5}{2^{k+2}}$ and the proof is complete.
\end{proof}

One might wonder whether the dyadic Hammersley point set is the digital $(0,m,2)$-net in base 2 with highest (asymptotical) dispersion. This is not the case, since the net $\cP_m^{L}$ generated by the matrices
\begin{equation*} \label{matrixc} C_1=J_m
\text{\, and\,}
 C_2=L_m=
\begin{pmatrix}
1 & 0 & 0 & \cdots & 0 & 0 & 0 \\
1 & 1 & 0 & \cdots &  0 & 0 & 0 \\
1 & 1 & 1 & \cdots & 0 & 0 & 0 \\
\vdots & \vdots & \vdots & \ddots & \vdots & \vdots & \vdots & \\
1 & 1 & 1 & \cdots &  1 & 0 & 0 \\
1 & 1 & 1 & \cdots &  1 & 1 & 0 \\
1 & 1 & 1 & \cdots &  1 & 1 & 1 \\
\end{pmatrix}.
\end{equation*}
has essentially larger dispersion. The following theorem further demonstrates that the constant $4$ in the upper bound on the dispersion of dyadic NLT and NUT nets given in Theorem~\ref{NUT}  is best possible.

\begin{thm} \label{badd}
  For all $m\geq 4$ we have 
	  $$ \disp(\cP_m^{L})\geq \max_{k=2,\dots,m-2}\left(\frac{2^{k+1}-1}{2^m}\left(\frac{1}{2^{k-1}}-\frac{1}{2^{m}}\right)\right)
		   =\begin{cases}
			      \frac{1}{2^m}\left(4-\frac{4}{2^{\frac{m}{2}}}+\frac{1}{2^m}\right) & \text{if $m$ is even,} \\
						\frac{1}{2^m}\left(4-\frac{3}{2^{\frac{m-1}{2}}}+\frac{1}{2^m}\right) & \text{if $m$ is odd.}
			  \end{cases} $$
\end{thm}

\begin{proof}
We note that $\cP_m^{L}$ can be written in the form
$$ \cP_m^{L}=\left\{\left(\frac{n}{2^m},\omega(n)\right):n=0,1,\dots,2^m-1\right\}, $$
where for $n=[e_m\dots e_1]_2$ we set $\omega(n)=\frac{e_1}{2}+\frac{e_1\oplus e_2}{2^2}+\dots+\frac{e_1\oplus \dots \oplus e_m}{2^{m}}. $
  We show that for every $k\in\{2,\dots,m-2\}$ the box
	  $$ J_k:=\left(\frac{3}{2^m},\frac{2^{k+1}+2}{2^m}\right)\times \left(\frac12-\frac{1}{2^k},\frac12+\frac{1}{2^k}-\frac{1}{2^m}\right) $$
		contains no element of $ \cP_m^{L}$, which proves the lower bound on the dispersion. 
		
		A point $(x,y)\in \cP_m^L$ can be in $J_k$ only if
		$2^m x \in \{4,\dots, 2^{k+1}+1\}$. We have to show that the corresponding $y$-coordinates are not in $\left(\frac12-\frac{1}{2^k},\frac12+\frac{1}{2^k}-\frac{1}{2^m}\right)$. We show: if $n\in \{4,\dots, 2^{k+1}+1\}$ is even, then $\omega(n)\leq \frac12-\frac{1}{2^k}$ and if $n$ is odd, then $\omega(n)\geq \frac12+\frac{1}{2^k}-\frac{1}{2^m}$. \\
		Clearly, $\omega(2^{k+1})=\sum_{j=k+2}^{m}2^{-j}<\frac18$ and $\omega(2^{k+1}+1)=\sum_{j=1}^{k+1}2^{-j}\geq \frac78$. Let $n\in \{4,\dots, 2^{k+1}-2\}$ be even; i.e. $n=2^ke_{k+1}+\dots+2e_2$ for some digits $e_i\in\{0,1\}$ for $i\in\{2,\dots,k+1\}$. Hence
		 $$ \omega(n)=\frac{e_2}{4}+\dots+\frac{e_2\oplus \dots \oplus e_k}{2^k}+\sum_{j=k+1}^{m}\frac{e_2\oplus \dots \oplus e_{k+1}}{2^j}. $$
		If $e_2\oplus \dots \oplus e_{k+1}=0$, then $\omega(n)\leq \sum_{j=2}^{k}2^{-j}=\frac12-2^{-k}$. Assume now $e_2\oplus \dots \oplus e_{k+1}=1$, which implies
		that the number of $i\in\{2,\dots,k+1\}$ such that $e_i=1$ is odd. If we assume additionally that $e_2=0$, then $\omega(n)\leq \frac14-2^{-m}\leq \frac12-2^{-k}$, since $k\geq 2$. Otherwise $e_2=1$. If $e_2$ were the only non-zero digit of $n$, then $n=2$, a contradiction. Hence there must be at least two more non-zero digits $e_i$ of $n$ with $i\in\{3,\dots,k+1\}$. Since the latter index set has at least two elements only if $k\geq 3$, we see that $e_2=1$ is impossible for $k=2$. If $k\geq 3$ and $e_2=1$, let $i_1<i_2$ minimal in $\{3,\dots,k+1\}$ such that the corresponding digits $e_{i_1}$ and $e_{i_2}$ are non-zero. Then $e_2\oplus \dots\oplus e_{i_1}=0$ and $3\leq i_1\leq k$. That yields $\omega(n)\leq \sum_{j=2}^{m}2^{-j}-2^{-i_1}\leq \frac12-2^{-k}$ and we are done with the case $n$ is even. \\
		Let $n\in \{5,\dots, 2^{k+1}-1\}$ be odd; i.e. $n=2^ke_{k+1}+\dots+2e_2+1$ for some digits $e_i\in\{0,1\}$ for $i\in\{2,\dots,k+1\}$. Hence
		 $$ \omega(n)=\frac{1}{2}+\frac{1\oplus e_2}{4}+\dots+\frac{1\oplus e_2\oplus \dots \oplus e_k}{2^k}+\sum_{j=k+1}^{m}\frac{1\oplus e_2\oplus \dots \oplus e_{k+1}}{2^j}. $$
		If $e_2\oplus \dots \oplus e_{k+1}=0$, then $\omega(n)\geq \frac12+\sum_{j=k+1}^{m}2^{-j}=\frac12+2^{-k}-2^{-m}$. Assume now that $e_2\oplus \dots \oplus e_{k+1}=1$. If additionally $e_2=0$, then $\omega(n)\geq \frac34\geq \frac12+2^{-k}$. Otherwise, $e_2$ cannot be the only non-zero digit of $n$, since then $n=3$, and we have $k\geq 3$. Hence we have indices $i_1<i_2$ minimal in $\{3,\dots,k+1\}$ such that the corresponding digits $e_{i_1}$ and $e_{i_2}$ are non-zero. Thus, $1\oplus e_2\oplus\dots\oplus e_{i_1}=1$ and $i_1\leq k$ and therefore $\omega(n)\geq \frac12+2^{-i_1}\geq \frac12+2^{-k}$. The proof is complete.	
\end{proof}

\section{Dispersion and discrepancy}

It is an interesting problem to investigate connections between discrepancy and dispersion of point sets. For instance, is it true that point sets in the unit square which achieve the optimal order of star discrepancy (or some other notion of discrepancy) also achieve the optimal order of dispersion? Does the reverse implication hold? Fibonacci lattices and $(0,m,2)$-nets indicate relations in this direction. \\
Let us compare our dispersion results with what we know about discrepancy of digital nets. We found that the point set $\cP_{m}^U$ has lower dispersion than the Hammersley point set. The same is true for the star and the $L_2$ discrepancy, respectively. Let us recall the definition of discrepancy for $N$-element point sets in the unit square. We define
the discrepancy function of such a point set $\cP$ by
$$ \Delta(\bst,\cP)=\sum_{\bsz\in\cP}\bsone_{[\bszero,\bst)}(\bsz)-Nt_1t_2, $$
where for $\bst=(t_1,t_2)\in [0,1]^2$ we set $[\bszero,\bst)=[0,t_1)\times [0,t_2)$ with volume $t_1t_2$
and denote by $\bsone_{[\bszero,\bst)}$ the indicator function of this interval. The star discrepancy of $\cP$ is given by
$$ D^*(\cP):=\sup_{\bst \in [0,1]^2}|\Delta(\bst,\cP)| $$
and the $L_2$ discrepancy  of $\cP$ is defined as
$$ L_{2}(\cP):=\left(\int_{[0,1]^2}|\Delta(\bst,\cP)|^2\rd \bst\right)^{\frac{1}{2}}. $$
While it is known that $\lim_{m\to\infty} \frac{D^*(\cH_{m,2})}{m}=\frac13$, we have $\limsup_{m\to\infty} \frac{D^*(\cP_{m}^U)}{m}\leq 0.2263...$ (see~\cite[Theorems 4 and 6]{Larch}). The difference between the corresponding $L_2$ discrepancies is even more significant, as they are not even of same order in $N=2^m$. We have $\lim_{m\to\infty} \frac{L_2(\cH_{m,2})^2}{m^2}=\frac{1}{64}$ and $\lim_{m\to\infty} \frac{L_2(\cP_{m}^U)^2}{m}= \frac{5}{192}$ (see~\cite{Kritz}). Note that these results demonstrate that both the Hammersley point set and the net $\cP_m^U$ achieve the optimal order of star discrepancy $\log{N}$ (this is true for all $(0,m,2)$-nets), but only $\cP_m^U$ achieves the optimal order of $L_2$ discrepancy $\sqrt{\log{N}}$.\\
However, while the Hammersley point set is the worst digital $(0,m,2)$-net with respect to star discrepancy, this is not the case for the dispersion, so there does not seem to be a direct connection between these quantities. While dispersion is a measure of irregularities of distribution with respect to arbitrary axes-parallel boxes, the star and $L_2$ discrepancy are defined with respect to anchored boxes. Therefore it seems more convenient to compare dispersion to extreme discrepancy, which is defined as
$$ D(\cP):=\sup_{B\in \mathcal{B}} \left|\sum_{\bsz\in\cP}\bsone_{B}(\bsz)-N\lambda(B)\right|. $$ 
 The obvious bound $D(\cP)\geq N\disp(\cP)$ does not yield any non-trivial results in our context.
Unfortunately, not many precise results are known for the extreme discrepancy. In particular, it is unknown whether the Hammersley point set has higher extreme discrepancy than the nets $\cP_m^U$ or $\cP_m^L$. For this reason, it is currently not possible to compare the respective dispersions and extreme discrepancies.

	\subsection*{Acknowledgments} The author is supported by the Austrian Science Fund (FWF),
Project F5509-N26, which is a part of the Special Research Program ``Quasi-Monte Carlo Methods: 
Theory and Applications''.

\bibliographystyle{plain}
\bibliography{dispnets}

\begin{thebibliography}{1}

\bibitem{Aist}
C.~Aistleitner, A.~Hinrichs, and D.~Rudolf.
\newblock On the size of the largest empty box amidst a point set.
\newblock {\em Discrete Appl. Math.}, 230(1):146--150, 2017.

\bibitem{Bren}
S.~Breneis and A.~Hinrichs.
\newblock Fibonacci lattices have minimal dispersion on the two-dimensional
  torus.
\newblock In D.~Bilyk, D.~Josef, and F.~Pillichshammer, editors, {\em
  Discrepancy Theory}, pages 117--132. De Gruyter, 2020.

\bibitem{Dick}
J.~Dick and F.~Pillichshammer.
\newblock {\em {Digital Nets and Sequences: Discrepancy Theory and Quasi-Monte
  Carlo Integration}}.
\newblock Cambridge University Press, 2010.

\bibitem{Dumi}
A.~Dumitrescu and M.~Jiang.
\newblock On the largest empty axes-parallel box amidst $n$ points.
\newblock {\em Algorithmica}, 66(2):225--248, 2013.

\bibitem{Kritz}
R.~Kritzinger and F.~Pillichshammer.
\newblock Digital nets in dimension two with the optimal order of {$L_p$}
  discrepancy.
\newblock {\em J. Th\'eor. Nombres Bordeaux}, 31(1):179--204, 2019.

\bibitem{Larch}
G.~Larcher and F.~Pillichshammer.
\newblock Sums of distances to the nearest integer and the discrepancy of
  digital nets.
\newblock {\em Acta Arith.}, 106(4):379--408, 2003.

\bibitem{Rote}
G.~Rote and R.~F. Tichy.
\newblock {Quasi-Monte Carlo} methods and the dispersion of point sequences.
\newblock {\em Math. Comput. Modelling}, 23(8--9):9--23, 1996.

\bibitem{Ullr}
M.~Ullrich and J.~Vyb\'{i}ral.
\newblock An upper bound on the minimal dispersion.
\newblock {\em J. Complexity}, 45:120--126, 2018.

\end{thebibliography}

\end{document}